\documentclass[11pt]{article} 
\usepackage[affil-it]{authblk}
\usepackage{setspace}
\usepackage[utf8]{inputenc} 

\usepackage{subfigure}
\usepackage[margin=1in]{geometry}
\usepackage{graphicx} 
\usepackage{epstopdf}

\usepackage{booktabs} 
\usepackage{array} 
\usepackage{paralist} 
\usepackage{verbatim} 
\usepackage{subfig} 

\usepackage{fancyhdr} 
\usepackage{sectsty}
\allsectionsfont{\sffamily\mdseries\upshape} 

\usepackage[nottoc,notlof,notlot]{tocbibind} 
\usepackage[titles,subfigure]{tocloft} 

\usepackage{amsmath, amsthm, amssymb}

\usepackage{multirow}
\usepackage{algorithm}
\usepackage[noend]{algorithmic}

\usepackage{amsthm}
\newtheorem{thm}{Theorem}
\newtheorem{lem}{Lemma}
\newtheorem{prop}{Proposition}
\newtheorem{cor}{Corollary}
\newtheorem{fact}{Fact}

\allowdisplaybreaks

\DeclareMathOperator*{\argmin}{argmin}
\usepackage{rotating}
\usepackage{pdflscape}

\usepackage{tikz,pgfplots}
\usetikzlibrary{matrix}
\usepgfplotslibrary{groupplots}
\pgfplotsset{compat=newest}
\usetikzlibrary{calc,positioning,shapes,fit,arrows,shadows}
\tikzstyle{line} = [ draw, -latex']
\usetikzlibrary{shapes,arrows}
\usepgfplotslibrary{units}
\usepackage{url}
\usepackage{color}

\graphicspath{{figures/}} 

\usepackage{color}

\newcommand{\cB}{\mathcal{B}}
\newcommand{\cL}{\mathcal{L}}
\newcommand{\cG}{\mathcal{G}}
\newcommand{\cS}{\mathcal{S}}

\usepackage{color}

\DeclareMathOperator*{\conv}{conv}

\DeclareMathOperator*{\extr}{extr}

\title{
Node-based valid inequalities for \\ the optimal transmission switching problem
}
\author{
Santanu S. Dey\thanks{santanu.dey@isye.gatech.edu, H. Milton  Stewart School of Industrial \& Systems Engineering, Georgia Institute of Technology, Atlanta, GA 30332.}, 
Burak Kocuk\thanks{burak.kocuk@sabanciuniv.edu, Industrial Engineering Program,  Sabanc{\i} University, Istanbul, Turkey 34956.},  
Nicole Redder\thanks{nredder@gatech.edu, H. Milton  Stewart School of Industrial \& Systems Engineering, Georgia Institute of Technology, Atlanta, GA 30332.}
}

\begin{document}

\maketitle

\vspace{-1cm}
\begin{abstract}
The benefits of transmission line switching are well-known in terms of reducing operational cost and improving system reliability of power systems. However, finding the optimal power network configuration is a challenging task due to the combinatorial nature of the underlying optimization problem. In this work, we identify a certain ``node-based'' set that appears as substructure of the optimal transmission switching problem and then conduct a polyhedral study of this set.
We construct an extended formulation of the integer hull of this set and present the inequality description of the integer hull in the original space in some cases. These inequalities in the original space can be used as cutting-planes for the transmission line switching problem. 
Finally, we present the results of our computational experiments using these cutting-planes on difficult test cases from the literature.
\end{abstract}

\section{Introduction}
The study of short-term electric power systems planning offers a wide range of interesting optimization problems for the operations research community \cite{kocuk2016global}. One of the fundamental optimization problems from this area is called the alternating current optimal power flow (AC-OPF) problem.
The main goal in this problem is to determine a minimum cost flow across the power network that satisfies demand while following physics laws such as the Ohm's law and Kirchhoff's law. However, the AC-OPF problem is very challenging to solve since the power flow constraints are nonlinear and nonconvex, the underlying networks are very large, and finally AC-OPF needs to be solved rapidly for changing power demands. Therefore, the constraints are typically approximated by direct current (DC) power equations, a linearization of the AC equations, as in \cite{Fisher}, \cite{Barrows12}, \cite{Fuller12},  \cite{Wu13}, and the resulting model is called as the DC-OPF problem. 

Because of the underlying physics of electricity, removing lines from a network may result in improved network efficiency. It is therefore possible that switching off lines may reduce the generation cost \cite{oneill.et.al:05}. In this way, we arrive at the optimal transmission switching (OTS) problem, formalized by Fisher et al. \cite{Fisher}. The OTS problem is the AC-OPF problem modified by the additional choice of switching lines off. With the DC approximation DC-OPF, OTS can be modeled as a mixed-integer linear program (MILP), in which binary variables model the switching status of the network lines; and power flow constraints are only applied to active lines.  Henceforth, in the paper we refer to the DC version of optimal switching problem as DC-OTS. 

DC-OTS is NP-hard to solve~\cite{lehmann2014complexity,kocuk2016cycle}. Since the linear relaxations are often weak (involving many big-M constraints), DC-OTS model with many binary variables are often very hard for modern computational integer programming software. To the best of our knowledge the only paper that conducts a polyhedral study of a substructure of DC-OTS in order to improve the quality of linear programming relaxation is~\cite{kocuk2016cycle}. We emphasize that finding high quality feasible solutions is also quite difficult (see, for example, a large number of heuristics proposed for this problem~\cite{hedman2010smart,liu2012heuristic, ruiz2012tractable,fuller2012fast,johnson2020k}). Therefore, the DC-OTS problem is a very challenging problem to solve exactly. 

In this paper, we identify a certain ``node-based'' set that appears as substructure of the DC-OTS problem and then conduct a polyhedral study of this set. We construct an extended formulation of the integer hull of this set. Then, we present the inequality description of the integer hull in the original space in some special (but commonly occurring in practice) cases. These inequalities in the original space can be used as cutting-planes for the DC-OTS problem. Finally, we present the results of our computational experiments using these cutting-planes on difficult test cases from the literature.


The rest of the paper is organized as follows: In Section~\ref{sec:ots} we formally introduce the MILP formulation of the DC-OTS problem. In Section~\ref{sec:substructure}, we identify the ``node-based'' set that appears as substructure of the DC-OTS problem and show that optimizing on this set is NP-hard. Then in Section~\ref{sec:ext}, we present an extended formulation of the integer hull of the ``node-based'' set and identify conditions under which this extended formulation is compact. Since the extended formulation is exponential in size in general, we present a compact outer approximation of the integer hull in Section~\ref{sec:app}. We then show that the outer approximation yields the convex hull in some special cases. In Section~\ref{sec:com}, we present out computational results. Finally, we make some concluding remarks in Section~\ref{sec:conc}.

\textbf{Notation}: Given a set $S$, we denote its convex hull, (affine) dimension, 
set of extreme points
as $\conv(S)$, $\dim(S)$,  
$\extr(S)$, 
respectively. We will use   $[m:n]$ as a shorthand notation for the set  $\{m, m+1, \dots,n\}$, where $m,n\in \mathbb{Z}$\textcolor{black}{, the set of integers}.

\section{DC optimal transmission switching problem}\label{sec:ots}

A power network is a set of nodes called {\it buses} and edges called {\it transmission lines} within a power system.  
Consider a power network $\mathcal{N} = (\mathcal{B},\mathcal{L})$, where $\mathcal{B}$ represents the node set and $\mathcal{L}$ represents the edge set. The subset of buses $\mathcal{G}\subseteq \mathcal{B}$ represents buses connected to generation units (power generators), and we assume that every bus holds some electric demand, called load.
The goal of DC-OTS is to satisfy demand at all buses with the minimum total production costs of generators such that the solution obeys the physical laws (e.g., Ohm's Law) and other operational restrictions (e.g. transmission line flow limit constraints).

We now describe the parameters of the DC-OTS problem:

\begin{itemize}
\item
For each bus $i \in \cB$, let $p_i^d$ and be the real power load.

\item
For each generator located at bus $i \in \cG$, active output is restricted to be in the interval $ [p_i^{\text{min}}  , p_i^{\text{max}}]$. We set $ p_i^{\text{min}} =  p_i^{\text{max}}=0$ for $i \in \mathcal{B} \setminus \mathcal{G}$.

\item For each transmission line $(i,j)\in\cL$, the susceptance  is denoted as $B_{ij}$ and  the thermal limit is denoted as $\overline s_{ij}$.

\end{itemize}

We  define the following decision variables to model the DC-OTS problem:
\begin{itemize}
\item
For each bus $i \in \cB$, let $\theta_i$ and be the phase angle.
\item
For each generator located at bus $i \in \cG$, let $p_i^g$ be its production output.
\item
For each transmission line $(i,j) \in \cL$, let $p_{ij}$ be its real power flow and $x_{ij}$ be its on/off status.
\end{itemize}
The DC-OTS problem is modeled as follows:
\begin{subequations} \label{Angle FormulationDC}
\begin{align}
  \min  &\hspace{0.5em}  \sum_{i \in \mathcal{G}} C_i(p_i^g)  \label{objDC} \\
  \mathrm{s.t.}   &\hspace{0.5em}  p_i^g-p_i^d = \sum_{j \in \delta(i) } p_{ij}    &i& \in \mathcal{B}   \label{activeAtBusDC} \\
 &  \hspace{0.5em}   p_{ij}=[B_{ij}(\theta_i - \theta_j) ]x_{ij}     &(&i,j) \in \mathcal{L} \label{powerOnArcDC} \\
  & \hspace{0.5em}  | p_{ij} | \le \overline s_{ij} x_{ij}    &(&i,j) \in \mathcal{L} \label{powerLimitArcDC} \\
\notag  & \hspace{0.5em}  p_i^{\text{min}}  \le p_i^g \le p_i^{\text{max}}    & i& \in \mathcal{B} \\
\notag  & \hspace{0.5em}  x_{ij}  \in\{ 0,1\} &(&i,j) \in \mathcal{L}
\end{align}
\end{subequations}
Here, objective function  \eqref{objDC} minimizes the total cost of production, where $C_i$ is a linear or convex quadratic cost function of generator $i$. Constraint \eqref{activeAtBusDC} enforces flow conservation at each bus $i$. Real power across line $(i,j)$, namely $p_{ij}$, is expressed in terms of the phase angles at buses $i$ and $j$ through constraint \eqref{powerOnArcDC}, 
which is an approximation of the true physical function. Finally, the flow limit on line $(i,j)$ is imposed by constraint  \eqref{powerLimitArcDC}. 

Note that constraint \eqref{powerOnArcDC} can be linearized using a big-M formulation to obtain a MILP formulation of DC-OTS as follows:
\begin{eqnarray*}
p_{ij} \leq [B_{ij}(\theta_i - \theta_j) ]   + 2\pi B_{ij}(1 - x_{ij})    \ (i,j) \in \mathcal{L} \\
 p_{ij} \geq [B_{ij}(\theta_i - \theta_j) ]   - 2\pi B_{ij}(1 - x_{ij})    \ (i,j) \in \mathcal{L}.
\end{eqnarray*}
Here, we use an upper bound of $2\pi$ on $\theta_i - \theta_j$.
%


\section{A ``node-based” substructure of DC-OTS}\label{sec:substructure}

In this section, we start our polyhedral study of a ``node-based” substructure of DC-OTS. 
Let $ \bar f_j \in \mathbb{R}_+$, $\underline f_j=-\bar f_j$, for $j\in[1:n]$, $\underline f_0 \le \bar f_0$ and $d \in \mathbb{R}$. For convenience, let  $x_0=1$.
Consider the set
\begin{equation}\label{eq:DCSubstruct}
\mathcal{S} = \bigg \{ (x,f) \in \{0,1\}^n \times \mathbb{R}^{n+1}: \sum_{j=0}^n f_j = d, \ \underline f_j x_j \le f_j \le \bar f_j x_j \ j \in [0:n]  \bigg \}.
\end{equation}
This mixed-integer set appears as a substructure in the DC-OTS problem as follows: Consider a bus in the power network with a generator\footnote{A bus without a generator can be simply modeled by setting $\underline f_0=\overline f_0=0$.} and $n$ incident transmission lines.  Let $d$ represent the load at node $0$, $f_0$ represent the dispatch variable with bounds $\underline f_0$ and $\overline f_0$,  $f_j$ represent the flow variable of  line $j$ with  a bound of $\bar f_j$, and $x_j$ represent the on/off status of line $j$, $j\in[1:n]$. See Figure~\ref{fig:illustration4} for an illustration.

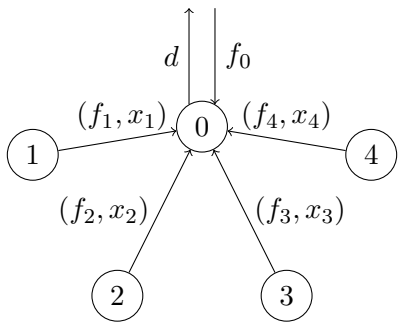
\begin{figure}[H]
\begin{center}
\begin{tikzpicture}[scale=0.75]
    \node[shape=circle,draw=black, label = { }] (A) at (0,0) {0};
    \node[shape=circle,draw=black] (B) at (-3,-0.5) {1};
    \node[shape=circle,draw=black] (C) at (-1.5,-3) {2};
    \node[shape=circle,draw=black] (D) at (1.5,-3) {3};
    \node[shape=circle,draw=black] (E) at (3,-0.5) {4};
    \node[shape=circle,draw=white] (dummy) at (0,2.5) {\color{white}{5}};

    \path [->] (B) edge node [align=center,above]{{ }$(f_1,
    {x_1})$} (A);
    \path [->] (C) edge node [align=center,left]{$(f_2, {x_2})$} (A);
    \path [->] (D) edge node [align=center,right]{$(f_3,{x_3})$} (A);
    \path [->] (E) edge node [align=center,above]{$(f_4, {x_4})$} (A);
    


\draw
(dummy.-60) edge[auto=left,->] node {$f_0$} (A.60) 
(A.120) edge[auto=left,->] node {$d$} (dummy.240)  ; ;

\end{tikzpicture}
\end{center}
\caption{Illustration with $n=4$.}\label{fig:illustration4}
\end{figure}

Our aim is to understand the complexity of optimizing a linear function over the set $\mathcal{S}$ and \textcolor{black}{to} study its  polyhedral properties (e.g. obtain its convex hull or a tight polyhedral  outer-approximation).

%
%
%

\begin{prop}\label{prop:DCSubstructCompl}
Optimizing a linear function over the set $\mathcal{S}$ defined in \eqref{eq:DCSubstruct} is NP-Hard.
\end{prop}
\begin{proof}
Consider the following decision problem, which we will call $\mathsf{DC-Node}$: Does there exist $(x,f)\in\mathcal{S}$ such that $\alpha^{\top} x + \beta^{\top} f \le \gamma$? We claim that $\mathsf{DC-Node}$ is NP-Complete, from which the statement of the  proposition follows.

We prove the NP-Completeness of  $\mathsf{DC-Node}$ by a reduction from $\mathsf{Subset-Sum}$, which is known to be NP-Complete \cite{Garey}. Consider an instance of $\mathsf{Subset-Sum}$ as follows: Given $a \in \mathbb{Z}^n_{++}$ and $b \in \mathbb{Z}_{++}$, does there exist a subset $J \subseteq [1:n]$ such that $\sum_{j \in J} a_j   = b$?  We construct an instance of $\mathsf{DC-Node}$ as below:
\[
\underline f_0=\overline f_0=0, \  \beta_0 = 0,\ d=b, \ \gamma=0, \ \overline f_i = -\underline f_i = a_i,  \ \alpha_i=a_i, \ \beta_i=-1 \quad i\in[1:n].
\]
We first note that the size of the  $\mathsf{DC-Node}$ instance is polynomial in the size of the  $\mathsf{Subset-Sum}$ instance.
We now verify that $\mathsf{DC-Node}$ is feasible if and only if $\mathsf{Subset-Sum}$ is feasible.

($\Rightarrow$): Let $(\hat x,  \hat f) \in \mathcal{S}$ such that $a^T \hat x - e^T \hat f \le 0$ be a solution to $\mathsf{DC-Node}$. Since $a_j \hat x_j - \hat f_j \ge 0$, we obtain that $\hat f_j = a \hat x_j$ for $j\in[1:n]$. Now, set $\hat J =\{j: \, \hat x_j = 1\}$. Since $\sum_{j\in\hat J} a_j = \sum_{j=1}^n \hat f_j = b$, the set $\hat J$ is a feasible solution to $\mathsf{Subset-Sum}$.

($\Leftarrow$): Let $\hat J \subseteq [1:n]$ be a solution to $\mathsf{Subset-Sum}$. Then, one can construct a feasible solution to $\mathsf{DC-Node}$ as
$
(\hat x_j, \hat f_j)= (1, a_j)
$
if $j \in \hat J$ and
$
(\hat x_j, \hat f_j)= (0,0)
$
otherwise.
\end{proof}

We now present an important sub-family of instances for which the optimization can be performed efficiently, since for this sub-family we will show the existence of a polynomial size extended formulation of the integer hull. The proof of Theorem \ref{prop:fixedKappaPolyTime} is presented in the next section.
\begin{thm}\label{prop:fixedKappaPolyTime}
Let $\bar\kappa$ be a fixed positive integer. Consider the set $\mathcal{S}$ defined in \eqref{eq:DCSubstruct} with $\underline f_0=\overline f_0 = 0$, 
$\bar f_j = \kappa_j \bar f$ where $\kappa_j\in[1:\bar\kappa]$ for $j\in[1:n]$, $\bar f \in \mathbb{R}_{++}$ and $d \in [0, \bar f)$. Then there exists an algorithm for optimizing a linear function over the set $\mathcal{S}$ whose running time is polynomial in the size of the input.
\end{thm}

\textcolor{black}{Although the requirements of Theorem~\ref{prop:fixedKappaPolyTime} might seem  restrictive at first glance, we note that it is quite common in a power network to have transmission line thermal limits which are multiples of each other. Therefore, the assumptions of Theorem~\ref{prop:fixedKappaPolyTime} are satisfied for the majority of the substructures that appear in  instances from the literature.  For example, in 118-bus instances from \cite{kocuk2016cycle},  the fraction of substructures with $\bar \kappa = 3$ is about 93\%  while  86\% of the substructures in 300-bus instances can be covered by $\bar \kappa = 7$.}


\section{Proof of Theorem~\ref{prop:DCSubstructCompl} via extended formulation of convex hull of $\mathcal{S}$}\label{sec:ext}

In this section, our aim is to obtain the convex hull of $\mathcal{S}$ as an extended formulation using disjunctive arguments.  
We start our analysis by identifying the extreme points of a certain restricted polyhedral set defined as
\begin{equation}\label{eq:DCSubstructGivenx}
\mathcal{S}(\hat x)= \bigg \{ f \in \mathbb{R}^{n+1} : \sum_{j\in J(\hat x)} f_j = d, \ \underline f_j  \le f_j \le \bar f_j \ j \in J(\hat x), \ f_j = 0 \ j \not\in J(\hat x)  \bigg\},
\end{equation}
where  $\hat x \in\{0,1\}^{n}$  and $J(\hat x) := \{0\}\cup \{j: \hat x_j = 1\}$.

\begin{prop}\label{prop:extrDCSubstructGivenx}
Consider the set $\mathcal{S}(\hat x)$ defined in \eqref{eq:DCSubstructGivenx}. Then, we have
\begin{equation*}
\begin{split}
\extr{(\mathcal{S}(\hat x)}) =  \bigcup_{k \in J(\hat x)} \bigg\{ f \in \mathbb{R}^{n+1} : & \ f_j \in \{\ \underline f_j, \bar f_j\}, \ j\in J(\hat x) \setminus \{k\}, \\ 
& \ f_k = d-\sum_{j\in J(\hat x) \setminus \{k\}} f_j,  \underline f_k  \le f_k \le \bar f_k , \\ 
& \ f_j = 0, \ j \not\in J(\hat x)  \bigg \}.
\end{split}
\end{equation*}
\end{prop}
\begin{proof}
Observe that $\mathcal{S}(\hat x)$ is a polyhedral set defined in $\mathbb{R}^{n+1}$. Therefore, at an extreme point, there must exist at least $n+1$ \textcolor{black}{linearly independent} inequalities satisfied at equality~\cite{wolsey1999integer}. Since the set  $\mathcal{S}(\hat x)$ has already one equality in its definition, at least $n$ of the bound constraints must be active. Therefore, enumerating over each coordinate gives the desired result.
\end{proof}
Now, we are ready to give the extreme points of  $\conv(\mathcal{S})$.
\begin{cor}\label{prop:extrDCSubstruct}
Consider the set $\mathcal{S}$ defined in \eqref{eq:DCSubstruct}. Then, we have
\begin{equation*}
\begin{split}
\extr{(\conv(\mathcal{S}))} =  \{ (x, f) \in \{0,1\}^n \times \mathbb{R}^{n+1} : f \in \extr{(\mathcal{S} (x)}) \}.
\end{split}
\end{equation*}
\end{cor} 
\begin{proof}
The result follows due to the fact that at any extreme point $(\hat x, \hat f)$ of $\conv(\mathcal{S})$, we must have $\hat x \in \{0,1\}^n$ and using Proposition \ref{prop:extrDCSubstructGivenx}.
\end{proof}

\subsection{An extended formulation for ${\conv(\mathcal{S})}$}
Below, we describe a way to obtain ${\conv(\mathcal{S})}$ through an extended formulation via a special layered network. 
To start with, let us define a set $V_k$, which contains the possible values the variable $f_k$ can take \textcolor{black}{between its bounds} in one of the extreme points of  $\conv(\mathcal{S})$, for $k\in[0:n]$. In particular, we have
\begin{equation*}
\begin{split}
V_k = \bigg\{f_k = d - \sum_{j \in [0:n]\setminus \{k\} }  f_j :   f_j \in E_j \ j\in[0:n] \setminus\{ k\}, \  \underline f_k  \le f_k \le \overline f_k  \bigg\},
\end{split}
\end{equation*}
for $k\in[0:n]$,
where
\[
E_j := \begin{cases}
\{\underline f_0, \overline f_0 \} & \text{ if } j=0 \\
\{-\bar f_j, 0, \bar f_j \} & \text{ if } j \in [1:n].
\end{cases}
\]
We note that each set $V_k$ can be obtained via enumeration.

We now describe a way to obtain ${\conv(\mathcal{S})}$ through an extended formulation via a special  network with ${n+2}$ layers, denoted as $N=(\cup_{i=-1}^{n} K_i, A)$. Here, $K_i$ is the set of nodes in layer $i$ and  $A$  is the set of arcs. 
\textcolor{black}{Layer $-1$ contains the source node and layer $n$ contains the sink nodes}. Each node in the network 
   is designated by a triplet $(i,\delta,S)$, where $i$ is the layer index, $\delta$ is the amount of demand satisfied until layer $i$ and $S$ is  the set of 
 coordinates  selected so far \textcolor{black}{that take value between their respective bounds}. The critical element in this construction is the fact that $|S| \le 1$ due to Proposition \ref{prop:extrDCSubstructGivenx} and Corollary \ref{prop:extrDCSubstruct}.  We use Algorithm \ref{alg:DCnetwork} to construct this network, \textcolor{black}{which we use  to represent the extreme points of ${\conv(\mathcal{S})}$ using source to sink paths.} See Figure \ref{fig:network} for an illustration.

\begin{algorithm}
\caption{Network construction for the extended formulation of $\conv(\cS)$.}
\label{alg:DCnetwork}
\begin{algorithmic}
\STATE Let $K_{-1} = \{(-1,0,\emptyset)\}$, $K_i=\emptyset$, $i=0,\dots,n$, and $A=\emptyset$. 
\FOR {$i=0,\dots,n$}

\STATE Set $\underline B_i=\max\{ \sum_{j=0}^i \underline f_j, d-\sum_{j=i+1}^n \bar f_j\}$
and $\overline B_i =\min\{\sum_{j=0}^i \bar f_j, d -\sum_{j=i+1}^n \underline f_j\} $.

\FOR {\textbf{each} $(i-1, \delta, S) \in K_{i-1}$}

\FOR {\textbf{each} $\vartheta \in E_i$}
\IF { $\underline B_i \le \delta+\vartheta \le \overline B_i$}
\STATE $K_i = K_i \cup \{ (i,\delta+\vartheta, S) \}$ and $A = A \cup \{ (i-1, \delta, S)  \to (i,\delta+\vartheta, S) \}$.
\ENDIF
\ENDFOR

\IF {$S = \emptyset$}
\FOR {\textbf{each} $\vartheta \in V_i$}
\IF { $\underline B_i \le \delta+\vartheta \le \overline B_i$}
\STATE $K_i = K_i \cup \{ (i,\delta+\vartheta, \{i\}) \}$ and $A = A \cup \{ (i-1, \delta, \emptyset)  \to (i,\delta+\vartheta, \{i\}) \}$.
\ENDIF
\ENDFOR

\ENDIF 
\ENDFOR
\ENDFOR

\end{algorithmic}
\end{algorithm}

\def\layersep{4.4cm}
\def\nodesep{1.4}

\begin{figure}[H]
\centering
\begin{tikzpicture}[shorten >=1pt,->,draw=black!50, node distance=\layersep, thick, scale=0.5]
    \tikzstyle{every pin edge}=[<-,shorten <=1pt]
    \tikzstyle{neuron}=[circle,fill=black!75,minimum size=15pt,draw=black,inner sep=0pt]
    \tikzstyle{input neuron}=[neuron, fill=green!0];
    \tikzstyle{output neuron}=[neuron, fill=red!0, minimum size=16pt,draw=white];
    \tikzstyle{hidden neuron}=[neuron, fill=blue!50];
    \tikzstyle{annot} = [text width=4em, text centered]

	\node[input neuron] (source) at (0,0) {\scriptsize{$-1,0,\emptyset$}};

    \foreach \name / \y in {-1,...,1}
        \node[input neuron] (I-\name) at (\layersep, \nodesep*\y) { };
        
       \node[output neuron] (I-0) at (\layersep,0) {{$\vdots$}};

    \foreach \name / \y in {-2,...,2}
        \path[yshift=0.0cm]
            node[input neuron] (H-\name) at (2*\layersep,\nodesep*\y) {};
            
       \node[output neuron] (H-0) at (2*\layersep,0) {{$\vdots$}};

	\node[output neuron] (dots) at (3*\layersep,0) {\Large{$\cdots$}};
            
    \foreach \name / \y in {-2,...,2}
        \path[yshift=0.0cm]
            node[input neuron] (K-\name) at (4*\layersep,\nodesep*\y) {};

       \node[output neuron] (K-0) at (4*\layersep,0) {{$\vdots$}};

    \foreach \name / \y in {-1,...,1}
        \node[input neuron] (L-\name) at (5*\layersep,\nodesep*\y) {};

       \node[output neuron] (L-0) at (5*\layersep,0) {{$\vdots$}};


%


%
    \node[annot,above of=source, node distance=2.1cm] (hl) {\footnotesize \textcolor{black}{Layer -1 (Source)}};
    \node[annot,above of=I-0, node distance=2.1cm] (hl) {\footnotesize Layer 0};
    \node[annot,above of=H-0, node distance=2.1cm] (hl) {\footnotesize Layer 1};
    \node[annot,above of=K-0, node distance=2.1cm] (hl) {\footnotesize Layer $n-1$};
    \node[annot,above of=L-0, node distance=2.1cm] (hl) {\footnotesize \textcolor{black}{Layer $n$ (Sinks)}};
    \node[annot,above of=dots, node distance=2.1cm] (hl) {$\cdots$};
\end{tikzpicture}
\caption{Illustration of the nodes of the layered network.}\label{fig:network}
\end{figure}
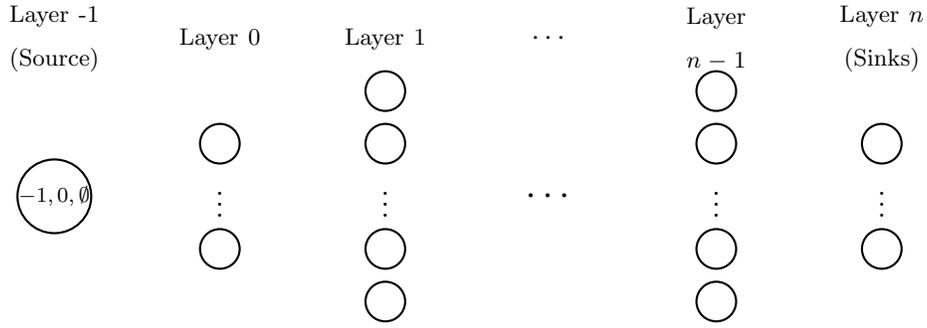


Let us define a set of binary variables $Y_{(i-1,\delta,S),(i,\delta',S')}$ which takes value 1 if the arc $(i-1,\delta,S)\to(i,\delta',S')$ is traversed, and 0 otherwise. Consider the following system:

\begin{subequations}\label{eq:DCnetworkForm}
\begin{align}
&\textcolor{black}{
\sum_{    \eta \in K_0 } Y_{(-1,0,\emptyset), \eta } = 1 
} \label{eq:DCnetworkForm-a} \\
& \textcolor{black}{
\sum_{ \eta\in K_{i-1} } Y_{\eta,\eta'} = \sum_{ \eta''\in K_{i+1} } Y_{\eta', \eta''} } \quad & {i} & \in[1:n-1], \textcolor{black}{ \eta'\in K_i }
 \label{eq:DCnetworkForm-b}  \\
&  \textcolor{black}{
x_i = 1-\sum_{(i-1,\delta,S)\in K_{i-1},(i,\delta,S)\in K_i} Y_{(i-1,\delta,S),(i,\delta,S)}} \quad &i&\in[1:n]  \label{eq:DCnetworkForm-c}  \\
& \textcolor{black}{
f_i = \sum_{(i-1,\delta,S)\in K_{i-1},(i,\delta',S')\in K_i} (\delta'-\delta) Y_{(i-1,\delta,S),(i,\delta',S')}} \quad &i&\in[0:n].  \label{eq:DCnetworkForm-d} 
\end{align}
\end{subequations}

\begin{prop}\label{prop:DCnetworkconv}
Consider the set $\mathcal{S}$ defined in \eqref{eq:DCSubstruct} and the layered network $N=(\cup_{i=-1}^{n} K_i, A)$ constructed using Algorithm \ref{alg:DCnetwork}. Then, we have that
\[
{\conv(\mathcal{S})} = \{ (x,f) \in[0,1]^n \times \mathbb{R}^{n+1} : \exists Y \ge 0 :  \eqref{eq:DCnetworkForm} \}.  
\]
\end{prop}
\begin{proof}
First, we note that  there is a one-to-one correspondence between  extreme points of $\conv(\mathcal{S})$ and source-to-sink paths in $N$.
Also, the constraint matrix of \eqref{eq:DCnetworkForm} is totally unimodular in $Y$. \textcolor{black}{This is due to the fact that constraints  \eqref{eq:DCnetworkForm-a}-\eqref{eq:DCnetworkForm-a} are defined by a network matrix, which is known to be totally unimodular, and constraints \eqref{eq:DCnetworkForm-c}-\eqref{eq:DCnetworkForm-d} are simply the definitions of the $x$ and $f$ variables in terms of the $Y$ variable.}  Hence, the result follows.
\end{proof}

\subsection{Proof of Theorem~\ref{prop:fixedKappaPolyTime}: A special case when the extended formulation is compact}

In the previous section, we  constructed an extended formulation for ${\conv(\mathcal{S})}$. However, this formulation can be of exponential-size in the worst case as optimizing a linear function over the set $\mathcal{S}$ is NP-Hard due to Proposition \ref{prop:DCSubstructCompl}. Now, we will consider an important sub-family of instances for which the size of the extended formulation is polynomial in $n$ with fixed parameter $\bar \kappa$.

\begin{prop}\label{prop:fixedKappaPolySize}
Under the assumptions of Theorem \ref{prop:fixedKappaPolyTime}, 
there exists a compact extended formulation for  ${\conv(\mathcal{S})}$ whose size is polynomial in $n$ and $\bar \kappa$.
\end{prop}
\begin{proof}
First of all, we can assume without loss of generality that $\bar f =1$ (otherwise, we can pass to new variables $f_j' = f_j / \bar f$). Consider the extended formulation \textcolor{black}{given in Proposition \ref{prop:DCnetworkconv}} constructed via Algorithm \ref{alg:DCnetwork}. It suffices to show that the number of nodes in the layered network is upper bounded by  a polynomial in $n$ and $\bar \kappa$.

Since $\underline f_j, \overline f_j \in \mathbb{Z}$, we have that $\delta$ is either an integer or a number of the form $p + d$ where $p \in \mathbb{Z}$. Notice also that $n \bar \kappa \leq \underline B_i := \textup{max} \{\sum_{j=0}^i \underline f_j, d - \sum_{j= i + 1}^n \overline f_j\}$, since $d \geq 0 $ and $- \underline f_j = \overline f_j = \kappa_j \leq \bar \kappa$. Similarly, $n \bar \kappa  + 1\geq \overline B_i := \textup{min} \{\sum_{j=0}^i \overline f_j, d - \sum_{j= i + 1}^n \underline f_j\}$. Thus $|K_i| \leq 4n\bar\kappa$. Therefore, $ \sum_{j=-1}^n|K_j| \le 4n^2\kappa$, which is a polynomial in $n$, and $\bar \kappa$ is a fixed integer.
\end{proof}

We are now ready to prove  Theorem \ref{prop:fixedKappaPolyTime}.
\begin{proof}[Proof of Theorem~\ref{prop:fixedKappaPolyTime}]
The statement of the proposition follows from the facts that i) optimizing a linear function over ${\mathcal{S}}$ is equivalent to optimizing the same function over ${\conv(\mathcal{S})}$, ii) there is a compact polyhedral extended formulation for ${\conv(\mathcal{S})}$ due to Proposition \ref{prop:fixedKappaPolySize}, and iii) linear programming is polynomially solvable.
\end{proof}

%
%

\section{An outer-approximation for convex hull of $\mathcal{S}$}\label{sec:app}
Since the size of the extended formulation for ${\conv(\mathcal{S} )}$ can be of exponential-size, we explore the possibility of obtaining a  compact outer-approximation in this section. 

\subsection{A compact extended formulation}
We start with the following proposition.
\begin{prop}\label{prop:outerDCSubstruct}
Consider the set $\mathcal{S}$ defined in \eqref{eq:DCSubstruct}, and define
\begin{equation}\label{eq:outerDCSubstruct}
\mathcal{O} =  \bigcup_{k=0}^n \bigg\{ (x,f) \in[0,1]^n \times \mathbb{R}^{n+1} : \sum_{j=0}^n f_j = d, \ \underline f_j x_j \le f_j \le \bar f_j x_j \ j\in[0:n], \ f_k \in \conv(V_k), \ x_k =1  \bigg\} .
\end{equation}
Then, we have $\conv(\mathcal{S}) \subseteq \conv(\mathcal{O})$.
\end{prop}
\begin{proof}
The result follows since the extreme points of $\conv(\mathcal{S})$ are contained in the set $\mathcal{O}$ by construction.
\end{proof}
We note that the size of the polyhedral representation of $\conv(\mathcal{O})$ is on the order of $O(n^2)$. 

\subsection{A special case when the outer-approximation is exact}
We will now discuss a special case for which i) $\conv(\mathcal{O}) = \conv(\mathcal{S})$ and ii) $\conv(\mathcal{S})$ can be described in the original space.

\begin{lem}\label{lem:intPolytope}
Let $(\bar f,d') \in \mathbb{R}_+\times \mathbb{R}$ such that   $d'=\kappa \bar f$ for some $\kappa \in [-m+1: m]$. Then, the following polytope is integral in $x$:
\[
\bigg\{ (x,f) \in[0,1]^m \times \mathbb{R}^{m} : d' -\bar f \le \sum_{j=1}^m f_j \le d', \ -\bar f x_j \le f_j \le \bar f x_j \ j\in[1:m]  \bigg\}.
\]
\end{lem}
\begin{proof}
First of all, we can assume without loss of generality that $\bar f_j =1$, $j=1,\dots,m$ (otherwise, we can pass to new variables $f_j' = f_j / \bar f$). Suppose that there is a fractional extreme point of the polytope, denoted by $(\hat x, \hat f)$, with $\hat x_i\in(0,1)$ for some $i$. We will look at two cases:

\noindent \underline{Case 1}. Suppose  $ \sum_{j=1}^m \hat f_j = d'$ holds (the case with $ \sum_{j=1}^m \hat f_j = d'-1$ can be handled similarly). First note that the point $(\hat x, \hat f)$ with $-\hat x_i < \hat f_i < \hat x_i$ cannot be extreme,
since it can be written as a convex combination of points with $\hat x_i \pm \epsilon$. Therefore, we must have that $\hat f_i = \pm \hat x_i$. Let us assume that $\hat f_i = \hat x_i$ (the case with $\hat f_i=-\hat x_i$ is similar). Since $d'$ is integer, there must exist another index $k$ such that $\hat f_k \in (-1,1)$. Now, we construct two new points $(x^\pm, f^\pm) = (\hat x \pm \epsilon e_i \mp \epsilon e_k, \hat f \pm \epsilon e_i \mp \epsilon e_k)$ with $\epsilon > 0$ small enough such that $(x^\pm, f^\pm)$ belong to the polytope (here, $e_i$ is the $i-$th unit vector). However, this contradicts to the fact that  $(\hat x, \hat f)$ is an extreme point. 

\noindent \underline{Case 2}. Suppose  $d' - \bar f <  \sum_{j=1}^m \hat f_j < d'$ holds. As before, we must have that $\hat f_i = \pm \hat x_i$.   Let us assume that $\hat f_i = \hat x_i$ (the case with $\hat f_i=-\hat x_i$ is similar).  Now, we construct two new points $(x^\pm, f^\pm) = (\hat x \pm \epsilon e_i ,  \hat f \pm \epsilon e_i )$ with $\epsilon > 0$ small enough such that $(x^\pm, f^\pm)$ belong to the polytope.
However, this contradicts to the fact that  $(\hat x, \hat f)$ is an extreme point.  
\end{proof}

\begin{prop}\label{prop:DCSubstructO=convS}
Consider the set $\mathcal{S}$ defined in \eqref{eq:DCSubstruct} with $\underline f_0=\overline f_0 = 0$, $\bar f_j = \bar f$ for $j\in[1:n]$, \textcolor{black}{$0 \leq$} $d < \bar f$, and the set $\mathcal{O}$ defined in \eqref{eq:outerDCSubstruct}.  Then,   we have $\conv(\mathcal{O}) = \conv(\mathcal{S})$.
\end{prop}
\begin{proof}
Since $\underline f_0=\overline f_0 = 0$, $\bar f_j = \bar f$ for $j\in[1:n]$,  we have that $V_k = \{ d, -\bar f +  d\}$ for $k\in[1:n]$.
Now, in order to prove that $\conv(\mathcal{O}) = \conv(\mathcal{S})$, it suffices to show that  $\conv(\mathcal{O})$  is integral in $x$. We first note that the following polytope in integral in $x$ for $\hat f_k \in V_k$ as a consequence of Lemma~\ref{lem:intPolytope}:
\[
\bigg\{ (x,f) \in[0,1]^n \times \mathbb{R}^{n+1} : \sum_{j=1}^n f_j = d, \ -\bar f x_j \le f_j \le \bar f x_j \ j\in[1:n], \ f_k \in [-\bar f +d,d], \ x_k =1  \bigg\}.
\]
Since $\conv(\mathcal{O})$ is the convex hull of the disjunctive union of integral polytopes (see Equation \eqref{eq:outerDCSubstruct}), it follows that $\conv(\mathcal{O})$  is integral as well, and the result follows.
\end{proof}

Note that the conditions under \textcolor{black}{which} the outer approximation $\textup{conv}(O)$ is equal to $\textup{conv}(\mathcal{S})$ \textcolor{black}{are} a special case of conditions presented in Theorem~\ref{prop:fixedKappaPolyTime} when we can show that the problem can be solved in polynomial-time. 

\subsection{The description of $\textup{conv}(\mathcal{S})$ in the original space under the special case}
In the setting of the previous subsection, it is also possible to obtain the convex hull description in the original set of variables. Let us define the following inequalities:
\begin{equation}\label{eq:fEqualIneqDesc}
\begin{split}
\sum_{j \in J_1} (\bar f - d) (\bar f x_j+f_j )  + \sum_{j \in J_2}  (\bar f - d) d x_j + \sum_{j \in J_3} d (\bar f x_j-f_j) &   \ge (\bar f - d) d  , 
\\ &  (J_1,J_2,J_3) \text{ is a partition of } [1:n].
\end{split}
\end{equation}

\begin{thm}\label{prop:fEqualIneqFacet}
Under the assumptions of Proposition~\ref{prop:DCSubstructO=convS}, the inequalities \eqref{eq:fEqualIneqDesc} are valid  for $\conv(\cS)$.
\end{thm}
\begin{proof}
We first prove the validity of inequalities \eqref{eq:fEqualIneqDesc} by checking them at each binary vector $\hat x \in \{0,1\}^n$. 
Let us define $J_i':=\{j: \ \hat x_j = 1,\textcolor{black}{j\in J_i}\}$, $i=1,2,3$. If $J_2' \neq \emptyset$, i.e. there exists $j^*\in J_2'$ such that $x_{j^*}=1$, then inequality  \eqref{eq:fEqualIneqDesc}  is trivially valid since the expressions $\bar f x_j+f_j$ and $\bar f x_j-f_j$ are nonnegative by the bound constraints. 
Now, suppose that $ J_2' = \emptyset$. In this case, it suffices to show that 
\begin{equation*}
\begin{split}
\sum_{j\in J_1'\cup J_3'} f_j = d, \ -\bar f \le f_j \le \bar f, \, j\in J_1'\cup J_3'
\implies & \sum_{j \in J_1'} (\bar f - d) (\bar f +f_j )  +  \sum_{j \in J_3'} d (\bar f -f_j) \ge  (\bar f - d) d  \\ 
\iff & \sum_{j \in J_1'}  f_j \ge  d  - (\bar f - d) |J_1'| - d |J_3'|  ,
\end{split}
\end{equation*}
where the  equivalence follows due to the linear equality. We will proceed by showing that 
\begin{equation}\label{eq:proofLP}
\min_{f} \bigg \{ \sum_{j\in J_1' } f_j : \sum_{j\in J_1'\cup J_3'} f_j = d, \ -\bar f \le f_j \le \bar f, \, j\in J_1'\cup J_3'  \bigg \} \ge 
d  - (\bar f - d) |J_1'| - d |J_3'|,
\end{equation}
by examining two cases.
\\
\underline{Case 1}. $|J_1'| \ge |J_3'|$: In this case, any optimal solution of \textcolor{black}{the optimization problem in}  \eqref{eq:proofLP} satisfies $\sum_{j\in J_3' } f_j^* = |J_3'| \bar f$ and  $\sum_{j\in J_1' } f_j^* = d - |J_3'| \bar f$. Hence, we obtain that
\[
\sum_{j\in J_1' } f_j^* = d - |J_3'| \bar f \ge  d - |J_3'| \bar f - (\bar f - d) (|J_1'| - |J_3'|)  = d  - (\bar f - d) |J_1'| - d |J_3'|.
\]
\underline{Case 2}. $|J_1'| < |J_3'|$: In this case, any optimal solution of \textcolor{black}{the optimization problem in} \eqref{eq:proofLP} satisfies $\sum_{j\in J_1' } f_j^* = -|J_1'| \bar f$ and  $\sum_{j\in J_3' } f_j^* = d + |J_1'| \bar f$. Hence, we obtain that
\[
\sum_{j\in J_1' } f_j^* = - |J_1'| \bar f \ge  -|J_1'| \bar f - d (|J_3'| - |J_1'| - 1)  = d  - (\bar f - d) |J_1'| - d |J_3'|.
\]


\end{proof}


\begin{thm}\label{prop:fEqualIneqDesc}
Let us define 
$$X:= \bigg\{ (x,f) \in[0,1]^n \times \mathbb{R}^{n+1} : \eqref{eq:fEqualIneqDesc}, \, f_0=0, \, \sum_{j=1}^n f_j = d, \ -\bar f_j x_j \le f_j \le \bar f_j x_j \ j\in[1:n] \bigg\}.$$ 
Under the assumptions of Proposition \ref{prop:DCSubstructO=convS}, we have $\conv(\cS) = X$.
\end{thm}
\begin{proof}


We first claim that $\textup{dim}(X) = \textup{dim}(\textup{conv}(\cS))=2n-1$.
In fact, let us consider the set of points 
\begin{align*}
    T= \bigg\{ (x,f): (f_j,x_j)=
    \begin{cases}
    (d,1), &j=i\\
    (0,0), &j\not =i
    \end{cases} , j\in [1:n] \bigg \}
    \cup 
     \bigg \{ (x,f): (f_j,x_j)=
    \begin{cases}
    (d,1), &j=i\\
    (1,1), &j=(i+1)\mod n\\
    (-1,1), &j=(i+2)\mod n
    \end{cases} \bigg \}.
\end{align*}
Since $|T|=2n$, the points of $T$ are affinely independent and contained in $\cS$ and $X$, we have that $\textup{dim}(\conv(\cS))\geq 2n-1$ and $\textup{dim}(X)\geq 2n-1$. As the description of $\cS$ and $X$ contains an equation, $\textup{dim}(\cS)\leq 2n-1$ and $\textup{dim}(X)\leq 2n-1$. Hence, we conclude that $\textup{dim}(X) = \textup{dim}(\textup{conv}(\cS))=2n-1$. 

Our proof strategy is to show that for any $c=[c^x, c^f]\in\mathbb{R}^{2n}$, \emph{all extreme optimal solutions} of the following problem 
\begin{equation}\label{eq:Mc}
    M(c) :=\argmin \bigg \{\sum_{j=1}^n (c^x_j x_j + c^f_j f_j) : (x,f) \in\textup{conv}(\cS) \bigg\},
\end{equation}
satisfy some inequality defining $X$ at equality. One of the key results we repeatedly use in the proof is the characterization of the extreme points of $\conv(S)$ established in Proposition~\ref{prop:extrDCSubstructGivenx} and Corollary~\ref{prop:extrDCSubstruct}, which implies the following: 
%
In an extreme solution, we have that $x_k \in \{0, 1\}$ for all $k \in [n]$, and there exists $ i\in [1:n]$ such that $ f_i\in \{d,-(1-d)\}$, $x_i=1$ and $f_j=\pm x_j$ for all $j\in [1:n]\setminus\{i\}$.


We start the proof with the following simple fact.
\begin{fact}\label{fact:fact c<0}
Suppose that the cost vector $[c^x,c^f]$ has an index $j\in[1:n]$ such that $c_j^x<0$. In this case, all optimal solutions satisfy $x_j=1$.
\end{fact}
Due to Fact~\ref{fact:fact c<0}, we  focus on the non-trivial case in which $c^x_j \geq 0$ for all $j \in [1:n]$ in the remainder of the proof.
Now, let us define 
\[
\mathcal{I}:=\argmin(c^x_i+c^f_i), \ \mathcal{J}:=\argmin(c^x_j-c^f_j), \ \text{and} \ z:=c^x_i+c^f_i+c_j^x-c^f_j, i\in \mathcal{I},j\in \mathcal{J}.
\] 
We have three cases based on the sign of $z$.

\noindent\underline{Case 1}. {$z> 0$}: \\
\underline{Claim 1} Consider any extreme point $(x,f)\in \conv(\cS)$ satisfying that there exist $i,j\in [1:n]$ such that $f_i = 1, f_j = -1$. We claim that such an extreme point is not optimal when $z>0$: Indeed, let $x^*_k=x_k,k\not\in \{i,j\}$ and $f^*_k=f_k,k\not\in \{i,j\}$; and $f_k=x_k=0, k\in \{i,j\}$. Then $(x^*,f^*)$ is feasible, and  observe that $$\sum_{i= 1}^n (c^x_i x^*_i + c^f_i f^*_i) + z \leq   \sum_{i = 1}^n (c^x_i x_i + c^f_i f_i) \Rightarrow  \sum_{i= 1}^n (c^x_i x^*_i + c^f_i f^*_i) < \sum_{i = 1}^n(c^x_i x_i + c^f_i f_i),$$ thus proving that $(x,f)$ is not an optimal solution.

Let $A=\argmin(c^x_i+dc^f_i), B = \argmin(c^x_i+c^f_i) = \mathcal{I}, C=\argmin(c^x_i-(1-d)c^f_i)$; and let $\tilde a, \tilde b, \tilde c$ be their respective values.
 Based on the above claim, it is easy to verify that all extreme points of $M(c)$ satisfy either: 
    
\begin{itemize}
    \item[(i)] Exactly one arc $i$ has 
    $x_i=1$, and $i\in A$. 
    \item[(ii)] Exactly two arcs $i,j$ exist such that their corresponding $x$ is positive, and $i\in B,j\in C$.
        
\end{itemize}
    
Now, we consider three sub-cases:

\begin{itemize}
    \item If $\tilde a<\tilde b+\tilde c$, then  we are in situation (i) above, i.e., exactly one arc has $x_i = 1$ and $i$ is in set $A$. Thus in any optimal solution  $(x,f)\in M(c)$ we have $\sum_i x_i=1$. 

    Therefore, we can assign $J_2=[1:n]$; the associated constraint from \eqref{eq:fEqualIneqDesc} 
    is satisfied at equality for all points in $M(c)$.
    
    \item If $\tilde b+\tilde c<\tilde a$. We first verify that in this case $B \cap C = \emptyset$. Assume by contradiction that there exists $i \in B \cap C$. Observe then that $$c^x_i+dc^f_i\geq \tilde a > \tilde b + \tilde c = c^x_i+c^f_i+c^x_i-(1-d)c^f_i\implies c^x_i<0,$$ a contradiction to our assumption that $c^x_i \geq 0$. Thus $B\cap C=\emptyset$. Now it is clear that the optimal solution is of the form (ii) above, i.e., exactly two arcs are active, one each from sets $B$ and $C$.  

    Assign $J_3=B,J_2=[1:n] \setminus B$, and it is straightforward to verify that the corresponding constraint from \eqref{eq:fEqualIneqDesc} is satisfied at equality for all points in $M(c)$.
    
    \item Otherwise, $\tilde a=\tilde b+\tilde c$. 

    \begin{itemize}
    
        \item If $B \cap C=\emptyset$, then we can assign $J_3=B, J_2=[1:n]\setminus B$.  If we have a solution of type (i) above with $x_i = 1$ and $f_i =d $ for some $i \in A$, note that whether $i \in J_2$ or $i\in J_3$, the inequality \eqref{eq:fEqualIneqDesc} is satisfied at equality. If we have a solution of type (ii) with $x_i, x_j  = 1$, $f_i = 1$ and $f_j = d - 1$ with $i \in B$ and $j \in C$, then we have that $i \in J_3$ and $j \in J_2$ and the inequality \eqref{eq:fEqualIneqDesc} is satisfied at equality.
            
        \item If $B \cap C \neq \emptyset$. In this case we first verify that $|B| = 1$. Let $i \in B \cap C$. Then 
                \begin{eqnarray}\label{eq:tocontradict}
        c^x_i+c^f_i+c^x_i-(1-d)c^f_i=2c^x_i+dc^f_i= \tilde b+\tilde c=\tilde a\leq c^x_i+dc^f_i\implies c^x_i=0.
        \end{eqnarray} This implies $\tilde b = c^x_i + c^f_i = c^f_i$. Now assume there exists $j \in B$ where $j \neq i$. Then 
        $0 < z \leq c^x_j + c^f_j + c^x_i - c^f_i = \tilde b + c^x_i  - \tilde b= c^x_i,$
        which contradicts (\ref{eq:tocontradict}).

        Again we can assign $J_3=B, J_2=[1:n]\setminus B$.  If we have a solution of type (i) above with $x_i = 1$ and $f_i =d $ for some $i \in A$, note that whether $i \in J_2$ or $i\in J_3$, the inequality \eqref{eq:fEqualIneqDesc} is satisfied at equality. If $|C| =1$, then we cannot have a solution of type (ii), since from the above claim we have $B= C$ with $|B| = |C| = 1$ (and therefore  any solution of type (ii) will have an objective function value strictly greater than $\tilde b + \tilde c$). If $|C| > 1$, then we have a solution of type (ii) with $x_i, x_j  = 1$, $f_i = 1$ and $f_j = d - 1$ with $i \in B$ and therefore crucially
        $j \in C\setminus B$. Then we have that $i \in J_3$ and $j \in J_2$ and thus the inequality \eqref{eq:fEqualIneqDesc} is satisfied at equality.
        
    \end{itemize}
\end{itemize}

\noindent\underline{Case 2}. \textbf{$z<0$}: \\
\underline{Claim 2} We first claim that, if $z<0$, then $\mathcal{I} \cap \mathcal{J} = \emptyset $. If not, then suppose $i \in \mathcal{I} \cap \mathcal{J}$. Then $$0 > z = c^x_i + c^f_i + c^x_i - c^f_i = 2c^x_i,$$
    which contradicts the assumption that $c^x_i \geq 0$. \\\\
We will consider two subcases: either $|\mathcal{J}| < |\mathcal{I}|$ or $|\mathcal{J}|\geq |\mathcal{I}|$.
\begin{itemize}

    \item Suppose $|\mathcal{J}| < |\mathcal{I}|$.  We will prove that $x_j = 1$ for all $j \in \mathcal{J}$ for any $(x,f) \in M(c)$. First observe that $\forall (x,f)\in M(c), \forall i\in \mathcal{I}, j\in \mathcal{J}, x_i+x_j\geq 1$ (otherwise, set both to one for a feasible solution with $f_i =1$ and $f_j = -1$ which has a strictly smaller objective value). Now suppose there exists $(x,f)\in M(c)$ such that  there exists $j\in \mathcal{J}$ with $x_j=0$ and $x_i=1$ for all $i\in \mathcal{I}$. Also let $\tilde{J}\subseteq \mathcal{J}$ such that $x_j = 1$ for $j \in \tilde{J}$. Then observe:
        
    \begin{itemize}
        
        \item There does not exist $i \in \mathcal{I}$ such that $f_i = -1$: If there exists $i \in \mathcal{I}$ such that $f_i = -1$, we obtain a better solution by setting $x_i = 0, f_i =0$ and $x_j = 1$ and $f_j = -1$. Thus, we cannot have $f_i = - 1$ for any $i \in \mathcal{I}$ for $(x,y) \in M(c)$.
            
        \item There does not  exist $i \in \mathcal{I}$ such that $f_i = d$: If there exists $i \in \mathcal{I}$ such that $f_i = d$ then (due to previous case) $f_u = 1$ for all $u \in \mathcal{I} \setminus \{i\}$. Since $|\tilde{J}| \leq |\mathcal{J}| - 1$, we have $|\tilde{J}| \leq |\mathcal{I}\setminus \{i\}| -1$. Thus, there exists $k\in [1:n]\setminus(\mathcal{I} \cup \mathcal{J})$ such that $ f_k= -1,x_k=1$, and $c^x_k-c^f_k>c^x_j - c^f_j$, contradicting optimality.
            
        \item There does not  exist $i \in \mathcal{I}$ such that $f_i = d -1$: Suppose there exists $i \in \mathcal{I}$ such that $f_i = d -1$. Based on previous case, $f_u = 1$ for all $u \in \mathcal{I} \setminus \{i\}$. Consider a new solution where we set $x_i = 1$, $f_i = d$, $x_j = 1$ $f_j = -1$ and all other variables have the same value. Then observe that the change in the objective function value is:
        $$c^x_i + dc^f_i + c^x_j - c^f_j - (c^x_i +  (d - 1)c^f_i) = c^x_j - c^f_j + c^f_i \leq c^x_j - c^f_j + c^x_i + c^f_i = z < 0.$$
        Thus, there does not  exist $i \in I$ such that $f_i = d -1$.
            
    \end{itemize}
    
From the above we have that for all $i \in \mathcal{I}$ we have $f_i = 1$. Now since $|\tilde{J}| \leq |\mathcal{J}| - 1$, we have $|\tilde{J}| \leq |\mathcal{I}| -2$. Thus, there exists $ k\in [1:n]\setminus(\mathcal{I} \cup \mathcal{J})$ such that $ f_k= -1,x_k=1$, and $c^x_k-c^f_k>c^x_j - c^f_j$, contradicting optimality. Thus $x_j = 1$ for all $j \in \mathcal{J}$.
    
Since $(x,f) \in M(c)$ we have $x_j = 1$ for all $j \in \mathcal{J}$. Thus select some $j \in \mathcal{J}$ and observe that $x_j \leq 1$ is satisfied at equality for all $(x,f) \in M(c)$. 

    \item Otherwise, $|\mathcal{I}|\leq |\mathcal{J}|$. We will prove that for all $i\in \mathcal{I}, x_i=1$, for all $(x,f)\in M(c)$.  
    Again it is true that $\forall (x,f)\in M(c), \forall i\in \mathcal{I}, j\in \mathcal{J}, x_i+x_j\geq 1$. 

Assume there exists an index $i \in \mathcal{I}$ 
such that $x_i  =0$. 
Then $x_j=1$, for all $j\in \mathcal{J}$. Let $\tilde{I}\subseteq \mathcal{I}$ such that $x_i = 1$ for $i \in \tilde{I}$.

Then observe:
    \begin{itemize}
    
        \item There does not exist $j\in \mathcal{J}$ such that $f_j=1$: Otherwise, we obtain a better solution by setting $x_j,f_j=0$ and $f_i,x_i=1$, contradicting optimality.
        
        \item There does not  exist $j \in \mathcal{J}$ such that $f_j = d -1$: If there exists $j \in \mathcal{J}$ such that $f_j = d$ then (due to previous case) $f_u = -1$ for all $u \in \mathcal{J} \setminus \{j\}$. Since $|\tilde{I}| \leq |\mathcal{I}| - 1$, we have $\tilde{I} \leq |\mathcal{J}\setminus \{j\}|$. Thus, there exists $k\in [1:n]\setminus(\mathcal{I} \cup \mathcal{J})$ such that $ f_k=1,x_k=1$, and $c^x_k+c^f_k>c^x_i+c^f_i$, contradicting optimality. 
        \item There cannot exist $j\in \mathcal{J}$ such that $f_j=d$: Suppose otherwise, for some $j$. Set $x_i=1,f_i=1,f_j=-(1-d)$. Then the objective function change is 
        \begin{align*}
            c^x_i+c^f_i+c^x_j-(1-d)c^f_j-c^x_j-dc^f_j=c^x_i+c^f_i-c^f_j=z-c^x_j<0,
        \end{align*}
        which contradicts optimality.
        
    \end{itemize}
    
    Finally, if for all $j \in J$, $f_j = -1$, then there exists $ k\in [1:n]\setminus(\mathcal{I} \cup \mathcal{J})$ such that $ f_k=1,x_k=1$, and $c^x_k+c^f_k>c^x_i+c^f_i$, contradicting optimality. Thus, for $(x,f) \in M(c)$, we have $x_i = 1$ for all $i \in \mathcal{I}$. 

Since $(x,f) \in M(c)$, we have $x_i = 1$ for all $i \in \mathcal{I}$ when $|\mathcal{J}|\geq |\mathcal{I}|$. 
Thus select some $i \in \mathcal{I}$ and observe that $x_i \leq 1$ is satisfied at equality for all $(x,f) \in M(c)$.
\end{itemize}
    
    
\noindent\underline{Case 3}. {$z=0$}: \\
\underline{Claim 3}\\
We claim that when $z=0$, $i\in \mathcal{I}\cap \mathcal{J}$ if and only if $c^x_i=0$. If there exists  $i\in \mathcal{I} \cap \mathcal{J}$, then $c^x_i+c^f_i+c^x_i-c^f_i=z =0$ implies $c^x_i=0$. If for some $i \in \mathcal{I} \cup \mathcal{J}$, $c^x_i=0$, then $c^f_i + c^x_i = c^f_i - c^x_i$, so $i\in \mathcal{I}\cap \mathcal{J}$.

\begin{itemize}

    \item Suppose $\mathcal{I}\cap \mathcal{J} =\emptyset$. This implies that $c^x_i>0$ for all $i\in \mathcal{I} \cup \mathcal{J}$ by Claim 3. Consider an extreme point $(x,f)$ of $M(c)$. We will show that there exists a partition of $[1:n]$ such that the inequality \eqref{eq:fEqualIneqDesc} is satisfied at equality.
    
    \begin{itemize}
        
        \item Firstly, for all $k\in \mathcal{I}$, we show $f_k\geq 0$:
        \begin{itemize}
            \item We will show that if there exists $j$ such that $f_{j}=-1$, then $j\in \mathcal{J}$. Suppose by contradiction that there exists $j^*\in [1:n] \setminus \mathcal{J}$ such that $f_{j^*}=-1$. Then because the set of possible flows at extreme points is $\{-1,-(1-d),0,d,1\}$ and $\sum_{i=1}^n f_i=d$, there must be some $i$ such that $f_i=1$. Because $J$ minimizes $c^x_j-c^f_j$ and $z=0$, $c^x_i+c^f_i+c^x_{j^*}-c^f_{j^*}>0$, and we could obtain a better solution by setting $x_{i},x_{j^*},f_{i},f_{j^*}=0$, contradicting optimality.
                
            \item We will show that there cannot exist $i\in \mathcal{I}$ such that $f_{i}=-(1-d)$. Suppose by contradiction that there exists $i^*\in \mathcal{I}$ such that $f_{i^*}=-(1-d)$. By flow conservation, there must exist some $j$ such that $f_j=1$; and $c^x_j+c^f_j\geq c^x_{i^*}+c^f_{i^*}$, since $\mathcal{I}$ minimizes that value. Suppose we set $x_j,f_j=0, f_{i^*}=d$. Then the objective value change is
            \begin{align*}
                &c^x_{i^*}+dc^f_{i^*}-(c^x_j+c^f_j+c^x_{i^*}-(1-d)c^f_{i^*})\\
                \leq & c^x_{i^*}+dc^f_{i^*}-(c^x_{i^*}+c^f_{i^*}+c^x_{i^*}-(1-d)c^f_{i^*})=-c^x_{i^*}<0,
            \end{align*}
            which contradicts optimality.
        \end{itemize}
        The two cases above show that $f_k\geq 0$ for all $k\in \mathcal{I}$.
        \item For all $k\in \mathcal{J}$, we show that $f_k\leq 0$:
        \begin{itemize}
            \item  We will show that if there exists $i$ such that $f_{i}=1$, then $i\in \mathcal{I}$. Suppose by contradiction that there exists $i^*\in [1:n] \setminus \mathcal{I}$ such that $f_{i^*}=1$.
            Firstly, we must have $f_i>0$ for all $i\in \mathcal{I}$; consider otherwise:
            \begin{itemize}
                \item If there exists $i\in \mathcal{I}$ such that $f_i=0$, then we could set $x_i=f_i=1, f_{i^*}=x_{i^*}=0$ for a better solution, contradicting optimality.
                \item We have already shown that for all $i\in \mathcal{I}$, $f_i\geq 0$. 
            \end{itemize}
            Then we have at minimum $1+d$ units of flow accounted for between $i^*$ and $\mathcal{I}$; so by flow conservation, there must exist $j\in [1:n]$ such that $f_j=-1$. As before, we can set $x_{j},x_{i^*},f_j,f_{i^*}=0$ and obtain a better solution, contradicting optimality.
            
            \item We will show that there cannot exist $j\in \mathcal{J}$ such that $f_j=d$. Suppose by contradiction that for some $j^*\in \mathcal{J}$, we have $f_{j^*}=d$.
            \begin{itemize}
                \item If there exists $i\in \mathcal{I}$ such that $f_i=0$, we can set $x_i=1,f_i=1,f_{j^*}=-(1-d)$. Then the objective function change is \begin{align*}
                    c^x_i+c^f_i+c^x_{j^*}-(1-d)c^f_{j^*}-c^x_{j^*}-dc^f_{j^*}=c^x_i+c^f_i-c^f_{j^*}=z-c^x_{j^*}<0,
                \end{align*}
                which contradicts optimality.
                
                \item Otherwise, we have $f_i=1$ for all $i\in \mathcal{I}$. Then by flow conservation, there must exist $j$ such that $f_j=-1$; and $c^x_i+c^f_i+c^f_j-c^f_j\geq z=0$. Then we can set $f_i,x_i,f_j,x_j=0$ for a feasible solution that is not worse; and we have already shown that such a solution cannot be optimal in the previous case. Then this solution cannot be optimal, either.
            \end{itemize}
        \end{itemize}
            The two cases above show that $f_k\leq 0$ for all $k\in \mathcal{J}$.
    \end{itemize}
    
    As we have already shown that $f_i=1\implies i\in \mathcal{I}$ and $f_i=-1\implies i\in \mathcal{J}$, then for any $i\in [1:n] \setminus(\mathcal{I}\cup \mathcal{J})$, $-1<f_i<1$.

    Assign $J_1 = \mathcal{J}, J_2=[1:n] \setminus(\mathcal{I}\cup \mathcal{J}),J_3=\mathcal{I}$, and consider the constraint \eqref{eq:fEqualIneqDesc}. Since $\mathcal{I}\subseteq J_3, \mathcal{J} \subseteq J_1$, when there exists $i$ such that $f_i=\pm 1$, the associated terms cancel out to zero. If there exists $i$ such that $f_i=d$, then $i\in J_2\cup J_3$, and equality is satisfied; otherwise, there exists $i$ such that $f_i=-(1-d)$, and $i\in J_1\cup J_2$, again satisfying equality.
    
    \item Otherwise there exists $i^*\in \mathcal{I} \cap \mathcal{J}$, and $c^x_{i^*}=0$. 
        
    We claim that for all $i$ with $c^x_i=0$, we have that $c^f_i=c^f_{i^*}$. For contradiction, suppose otherwise. If $c^f_i<c^f_{i^*}$, then $i^*\in \mathcal{I}$ is contradicted. If $c^f_i>c^f_{i^*}$, then $i^*\in \mathcal{J}$ is contradicted. 
        
    We will consider some cases and show that in each of the cases
    either we can find an inequality describing a facet of $X$ which all extreme points of $M(c)$ lie on, or else $M(c)=\cS$.  Before presenting the cases, we require the following claim. 
    
    \underline{Claim 4}\\
    When $z=0$ and $\mathcal{I} \cap \mathcal{J} \not = \emptyset$, if $c^x_i>0$ for some $i$, then $f_i\in \{-1,0,1\}$. Consider an $i$ such that $c^x_i>0$.  We will show that $f_i\not= d$. Suppose by contradiction that $f_i=d$, and note that $c^x_i+c^f_i\geq c^f_{i^*}$, so $c^x_i+dc^f_i> dc^x_{i^*}$.
            \begin{itemize}
                \item If $f_{i^*}=0$, we can assign $f_i,x_i=0$ and $f_{i^*}=d$ for a strictly better solution, contradicting optimality.
                \item If $f_{i^*}=1$, then by flow conservation there must exist  $j$ such that $f_j=-1$. We can set $f_j,x_j,f_{i^*},x_{i^*}=0$, for an objective change not worse than $z=0$. Then we can set $f_i,x_i=0$ and $f_{i^*}=d$ for a strictly better solution, contradicting the optimality of the unmodified solution.
                \item If $f_{i^*}=-1$, then by flow conservation there must exist  $k$ such that $f_k=1$, and we can set $f_k,x_k,f_{i^*},x_{i^*}=0$, for an objective change not worse than $z=0$. Then we can set $f_i,x_i=0$ and $f_{i^*}=d$ for a strictly better solution, contradicting the optimality of the unmodified solution.
            \end{itemize}
            
    We next show that $f_i\not = -(1-d)$. Suppose by contradiction that $f_i=-(1-d)$, and note that $c^x_i-c^f_i\geq -c^f_{i^*}$, so $c^x_i-(1-d)c^f_i> -(1-d)c^f_{i^*}$.
            \begin{itemize}
                \item If $f_{i^*}=0$, we can assign $f_i,x_i=0$ and $f_{i^*}=-(1-d)$ for a strictly better solution, contradicting optimality.
                \item If $f_{i^*}=-1$, then by flow conservation there must exist $k$ such that $f_k=1$, and we can set $f_k,x_k,f_{i^*},x_{i^*}=0$, for an objective change not worse than $z=0$. Then we can set $f_i,x_i=0$ and $f_{i^*}=-(1-d)$ for a strictly better solution, contradicting the optimality of the unmodified solution.
                \item If $f_{i^*}=1$, we can assign $f_i,x_i=0$ and $f_{i^*}=d$; the resulting objective change is 
                $dc^f_{i^*}-(c^f_{i^*}+c^x_i-(1-d)c^f_i)=-(1-d)c^f_{i^*}-(c^x_i-(1-d)c^f_i)<0$,
                contradicting optimality.
            \end{itemize}
            
            The above two cases prove that if $c^x_i>0$ for some $i$, then $f_i\in \{-1,0,1\}$.
    Now, we will consider some cases that are analyzed using Claim 4.
    \begin{itemize}
        \item If there exists $i\in [1:n] \setminus(\mathcal{I} \cup \mathcal{J})$, then we will show $x_i=0$ for all $(x,f)\in M(c)$:\\
        Suppose there exists such an $i$ and $f_i\not = 0$ for some extreme point of $M(c)$.
        \begin{itemize}
            \item $f_i\in \{-1,0,1\}$ since $c^x_i>0$.
            \item If $f_{i}=-1$, then by flow conservation there must exist $k$ such that $f_k=1$; and $c^x_k+c^f_k+c^x_i-c^f_i>0$ since $i\not \in J$. Then we can set $f_k,x_k,f_i,x_i=0$ to obtain a better solution, contradicting optimality.
            \item If $f_i=1$, and there exists $k$ such that $f_k=-1$, $c^x_i+c^f_i+c^x_k-c^f_k>0$ since $i\not \in \mathcal{I}$. Then we can set $f_k,x_k,f_i,x_i=0$ to obtain a better solution, contradicting optimality. Otherwise, by flow conservation, there exists some $k$ such that $f_k=-(1-d)$, and all other arcs carry no flow. For all $q\not \in \mathcal{I}\cap \mathcal{J}$ we have $c^x_q>0$, and $-c^f_{i^*}\leq c^x_q-c^f_q$, then $-(1-d)c^f_{i^*}\leq (1-d)c^x_q-(1-d)c^f_q<c^x_q-(1-d)c^f_q$ and we must have $k\in \mathcal{I} \cap \mathcal{J}$; else we have $f_{i^*}=0$ and we can swap $i^*$ and $k$ for a better solution, contradicting optimality. Then $c^x_k=0$ and $c^f_k=c^f_{i^*}$. Then we can set $f_i,x_i=0$ and $f_k=d$, and the change in objective is $dc^f_{i^*}-(c^x_i+c^f_i-(1-d)c^f_{i^*})=c^f_{i^*}-c^x_i-c^f_i<0$ since $i\not \in \mathcal{I}$, contradicting optimality.
        \end{itemize}
        Then $f_i=0$ in all optimal solutions. Since $c^x_i>0$, $x_i=0$ in all optimal solutions as well. 
        Thus, $x_i \geq 0$ is satisfied at equality for all $(x,f) \in M(c)$.

        Now suppose no such $i$ exists for the remainder of the proof.
        
        \item If there exists $i\in \mathcal{I}\cup \mathcal{J}$ such that $c^x_i>0$, then we will show that we have either $f_i=x_i$ or $f_i=-x_i$, for all $(x,f)\in M(c)$.
        \begin{itemize}
            \item If $i\in \mathcal{I}$  and $c^x_i>0$, then we will show that $f_i=x_i$. \\
            We will show that $f_i\in \{0,1\}$. Since $c^x_i>0$, we must have $f_i\in \{-1,0,1\}$. Suppose by contradiction that $f_i=-1$. By flow conservation, there must exist $k$ such that $f_k=1$. Since $i\not\in \mathcal{J}$ and $z=0$, we can set $f_i,x_i,f_k,x_k=0$ and obtain a better solution, contradicting optimality.\\
            Then $f_i\in \{0,1\}$ and if $f_i=0$ then $x_i=0$ since $c^x_i>0$, so $f_i=x_i$.
            
            \item If $i\in \mathcal{J}$, then we will show that $f_i=-x_i$. \\
            First, we show that $f_i\in \{-1,0\}$. Since $c^x_i>0$, we must have $f_i\in \{-1,0,1\}$. Then suppose by contradiction that $f_i=1$. 
            
            \begin{itemize}
                \item If $f_{i^*}>0$, then by flow conservation there must exist $k$ such that $f_k=-1$. Then since $i\not\in \mathcal{I}$, we can set $f_i,x_i,f_k,x_k=0$ for a better solution, contradicting optimality.
                \item If $f_{i^*}=0$, then since $i\not \in \mathcal{I}$ and $z=0$, we can set $f_{i^*}=1$ and $x_i,f_i=0$ for a better solution, contradicting optimality.
                \item If $f_{i^*}=-1$, then since $i\not \in \mathcal{I}$ and $z=0$, we can set $f_{i^*},f_i,x_i=0$ for a better solution, contradicting optimality.
                \item If $f_{i^*}=-(1-d)$, then we can set $f_{i^*}=d$ and $f_i,x_i=0$, and the objective value change is $dc^f_{i^*}-(c^x_i+c^f_i-(1-d)c^f_{i^*})=c^f_{i^*}-c^x_i-c^f_i<0$ since $i\not \in \mathcal{I}$, contradicting optimality.
            \end{itemize}
            Then $f_i\in \{0,-1\}$ and if $f_i=0$ then $x_i=0$ since $c^x_i>0$, so $f_i=-x_i$. Suppose no such $i$ exists for the remainder of the proof.
        \end{itemize}

        \item Then otherwise $I\cap J=[1:n]$ and there exists some $\tilde C$ such that $c^x_i=0,c^f_i=\tilde C,\forall i\in [1:n]$. Then due to flow balance, $M(c)=\cS$.
    \end{itemize}
\end{itemize}
\end{proof}

\begin{prop}\label{prop:fEqualIneqDescSep}
Inequalities \eqref{eq:fEqualIneqDesc} can be separated in linear time.
\end{prop}
\begin{proof}
 We now show that the separation of inequalities \eqref{eq:fEqualIneqDesc} can be performed in polynomial time. Given a point $(\hat x, \hat f)$, we define $\hat \alpha_j^1 := (\bar f - d) (\bar f \hat x_j+\hat f_j )$, $\hat \alpha_j^2 := (\bar f - d) d \hat x_j $ and $\hat \alpha_j^3 := d (\bar f \hat x_j- \hat f_j) $, $j=1,\dots,n$, and construct the partition $(\hat J_1, \hat J_2, \hat J_3) $ such that
 \[
 \hat J_1 := \{j: \hat \alpha_j^1\le \hat \alpha_j^2,  \, \hat \alpha_j^1\le \hat \alpha_j^3\}, \
 \hat J_2 := \{j: \hat \alpha_j^2 <  \hat \alpha_j^1, \,  \hat \alpha_j^2 \le \hat \alpha_j^3\}, \text{ and }
 \hat J_3 := \{j: \hat \alpha_j^3 <  \hat \alpha_j^1, \,  \hat \alpha_j^3 < \hat \alpha_j^2\}. 
 \]
 If $\sum_{i=1}^3\sum_{j \in \hat J_i} \hat \alpha_j^i < (\bar f - d)d$, then this partition corresponds to a most violated inequality and can be obtained in linear time with respect to $n$.
\end{proof}

 \section{Computational results}\label{sec:com}

In this section, we explain our computational setting and present our computational results that compare various approaches to solve the DC-OTS problem for challenging instances from the literature.

\subsection{Instances}
In our computational experiments, we use the instances presented in~\cite{kocuk2016cycle}. There are five groups of 35 instances each and some statistics are reported in Table~\ref{tab:instance}.
In each instance, there are $|\mathcal{B}|$ many binary variables and $|\mathcal{B}| +|\mathcal{G}| + |\mathcal{L}|$ many continuous variables.

 \begin{table}[H]\small
 \centering
\begin{tabular}{c|ccc}

  Instance &        $|\mathcal{B}|$ &       $|\mathcal{L}|$ &        $|\mathcal{G}|$ \\
\hline
       118Blumsack\_15 &        118 &        186 &         19 \\

       118Blumsack\_9G &        118 &        186 &         19 \\

       118Blumsack\_5p &        118 &        191 &         19 \\

       118Blumsack\_15p &        118 &        191 &         19 \\

       300New\_5 &        300 &        411 &         61 \\
\hline
\end{tabular}  
 \caption{Instance statistics.}
 \label{tab:instance}
 \end{table}
 We refer the reader to \cite{kocuk2016cycle} for the details about these instances. 
 
\subsection{Computational setting}

In our experiments, we use the commercial MILP solver CPLEX 12.8   on a 64-bit personal computer with Intel Core i7 CPU 2.60GHz processor (16 GB RAM).
A time limit of 3600 seconds is given for all settings and the relative optimality gap is set to $0.1\%$.
We compare the following four settings in our computational experiments:
 \begin{itemize}
 \item  Default: Default CPLEX.
 \item  Feasibility Oriented: Default CPLEX except \texttt{MIPEmph=4} and \texttt{DiveType=2}. As we mentioned earlier, finding feasible solution is also challenging and we will see that changing CPLEX to this setting improves its performance.
 \item  Feasibility Oriented setting and cuts separated from the convex hull defined in \eqref{eq:DCnetworkForm} in the support of $x$ and $f$ variables. The cuts are separated using a cut generation linear program. Five rounds of cuts are separated at the root node. 
 \item  Feasibility Oriented setting and inequalities \eqref{eq:fEqualIneqDesc} are separated. Notice that \eqref{eq:fEqualIneqDesc} can always be used by relaxing the upper bounds $\bar f$ in the set $\mathcal{S}$ (corresponding to a node).  \textcolor{black}{Only the nodes without a generator are considered in this setting due to the assumption of Theorem~\ref{prop:fEqualIneqDesc}.} Again five rounds of cuts are separated at the root node.  
 \end{itemize} 

\subsection{Computational results} 

We present our computational  results  in Tables~\ref{tab:default+feas}-\ref{tab:conv}. 
The following abbreviations are used in the tables:
\begin{itemize}
\item OptTime: Time in seconds used by CPLEX to prove optimality (except time for cut separation).
\item \#Nodes: Number of nodes of the branch-and-bound tree
    \item SepTime: Time in seconds needed to separate the cuts.
    \item TotalTime: Total time in seconds to prove optimality.
\end{itemize}

Table~\ref{tab:default+feas} shows results comparing Default CPLEX with Feasibility Oriented CPLEX. Table \ref{tab:facets} shows the performance of Feasibility Oriented CPLEX together with cuts separated from the extended formulation. Table \ref{tab:conv} shows the performance of Feasibility Oriented CPLEX together with \eqref{eq:fEqualIneqDesc} cuts separated.

 \begin{table}[H]\small
 \centering
 \begin{tabular}{rr|rrr|rrr|}

           &            &        \multicolumn{ 3}{c|}{Default} &     \multicolumn{ 3}{c|}{Feasibility Oriented} \\

           &            &   Unsolved &    OptTime &     \#Nodes &   Unsolved &    OptTime &     \#Nodes \\
 \hline
 \multicolumn{1}{c}{\multirow{ 2}{*}{118Blumsack\_15}} &         GA & \multicolumn{1}{c}{\multirow{ 2}{*}{0}} &       5.97 &      19825 & \multicolumn{1}{c}{\multirow{ 2}{*}{0}} &       6.03 &      11236 \\

 &         AA & \multicolumn{ 1}{c}{} &     100.67 &     586191 &  &       7.69 &      18461 \\
 \hline
 \multicolumn{1}{c}{\multirow{ 2}{*}{118Blumsack\_9G}} &         GA & \multicolumn{1}{c}{ \multirow{ 2}{*}{7} }&      32.78 &     180040 & \multicolumn{1}{c}{\multirow{ 2}{*}{5} }&      26.10 &      76669 \\

 &         AA & \multicolumn{ 1}{c}{} &     686.52 &    4943363 &  &     647.66 &    2825909 \\
 \hline
 \multicolumn{1}{c}{\multirow{ 2}{*}{118Blumsack\_5p}} &         GA &\multicolumn{1}{c}{\multirow{ 2}{*}{0}} &       7.33 &      26880 &\multicolumn{1}{c}{ \multirow{ 2}{*}{0}} &       8.06 &      16441 \\

  &         AA & \multicolumn{ 1}{c}{} &      19.96 &     149245 &  &      10.32 &      30145 \\
\hline
 \multicolumn{1}{c}{\multirow{ 2}{*}{118Blumsack\_15p}} &         GA &\multicolumn{1}{c}{ \multirow{ 2}{*}{5}} &     147.36 &     527989 &\multicolumn{1}{c}{ \multirow{ 2}{*}{4} }&      83.56 &     254854 \\

 &         AA & \multicolumn{ 1}{c}{} &     757.16 &    2484548 & &     481.14 &    1169934 \\
 \hline
 \multicolumn{1}{c}{\multirow{ 2}{*}{300New\_5}} &         GA & \multicolumn{1}{c}{\multirow{ 2}{*}{12}} &     192.72 &     234234 & \multicolumn{1}{c}{\multirow{ 2}{*}{6}} &      78.99 &      88550 \\

 &         AA & \multicolumn{ 1}{c}{} &    1284.98 &    1449385 & &     665.87 &     698018 \\
 \hline
 \end{tabular}  
 \caption{DC-OTS Results without cuts (GA: Geometric Average, AA: Arithmetic Average).}
 \label{tab:default+feas}
 \end{table}

 \begin{table}[H]\small
 \centering
 \begin{tabular}{rr|rrrrrr|}

           &            &   Unsolved &    SepTime &      \#Cuts &    OptTime &     \#Nodes &  TotalTime \\
 \hline
 \multicolumn{1}{c}{\multirow{ 2}{*}{118Blumsack\_15}} &         GA & \multicolumn{ 1}{c}{\multirow{ 2}{*}{0}} &       0.55 &        187 &       6.99 &      12048 &       7.71 \\

  &         AA & \multicolumn{ 1}{c}{} &       0.56 &        187 &      29.12 &      76788 &      29.68 \\
\hline
 \multicolumn{1}{c}{\multirow{ 2}{*}{118Blumsack\_9G}}&         GA & \multicolumn{ 1}{c}{\multirow{ 2}{*}{2}} &       0.54 &        185 &      24.76 &      69342 &      26.26 \\

&         AA & \multicolumn{ 1}{c}{} &       0.54 &        185 &     356.82 &    1443889 &     357.36 \\

 \hline
 \multicolumn{1}{c}{\multirow{ 2}{*}{118Blumsack\_5p}} &         GA & \multicolumn{ 1}{c}{{\multirow{ 2}{*}{0}}} &       0.65 &        192 &       8.81 &      16811 &       9.56 \\

  &         AA & \multicolumn{ 1}{c}{{\it }} &       0.65 &        192 &      11.54 &      31664 &      12.19 \\
\hline
 \multicolumn{1}{c}{\multirow{ 2}{*}{118Blumsack\_15p}} &         GA & \multicolumn{ 1}{c}{\multirow{ 2}{*}{1}} &       0.58 &        186 &      41.44 &     119906 &     42.17 \\

 &         AA & \multicolumn{ 1}{c}{} &       0.58 &        186 &     146.05 &     368832 &     146.64 \\
 \hline
 \multicolumn{1}{c}{\multirow{ 2}{*}{300New\_5}}&         GA & \multicolumn{ 1}{c}{{\multirow{ 2}{*}{4}}} &       2.17 &        455 &      91.54 &      89363 &      95.93 \\

  &         AA & \multicolumn{ 1}{c}{{\it }} &       2.18 &        455 &     523.17 &     485283 &     525.35 \\
\hline
 \end{tabular}  
 \caption{DC-OTS Results with Feasibility Oriented setting and cuts separated from the convex hull.}
  \label{tab:facets}
 \end{table}

 \begin{table}[H]\small
 \centering

 \begin{tabular}{rr|rrrrrr|}

           &            &   Unsolved &    SepTime &      \#Cuts &    OptTime &     \#Nodes &  TotalTime \\
\hline
\multicolumn{1}{c}{\multirow{ 2}{*}{118Blumsack\_15}}&         GA & \multicolumn{ 1}{c}{{\multirow{ 2}{*}{0}}}  &       0.02 &        148 &       5.26 &       8250 &       5.28 \\
 
 &         AA & \multicolumn{ 1}{c}{} &       0.02 &        148 &       6.06 &      12511 &       6.08 \\
\hline
\multicolumn{1}{c}{\multirow{ 2}{*}{118Blumsack\_9G}}&         GA & \multicolumn{ 1}{c}{{\multirow{ 2}{*}{2}}}  &       0.01 &        139 &      28.57 &      76579 &      28.62 \\

  &         AA & \multicolumn{ 1}{c}{} &       0.02 &        139 &     422.73 &    1692401 &     422.74 \\
\hline
\multicolumn{1}{c}{\multirow{ 2}{*}{118Blumsack\_15p}}&         GA & \multicolumn{ 1}{c}{{\multirow{ 2}{*}{0}}}  & {0.02} &  {141} & {8.13} & {16297} & {8.15} \\

&         AA & \multicolumn{ 1}{c}{{}} & {0.02} &  {141} & {10.18} & {33747} & {10.20} \\
\hline
\multicolumn{1}{c}{\multirow{ 2}{*}{118Blumsack\_15p}}&         GA & \multicolumn{ 1}{c}{{\multirow{ 2}{*}{2}}}  &       0.02 &        137 &      55.90 &     153063 &      55.94 \\

&         AA & \multicolumn{ 1}{c}{} &       0.02 &        137 &     282.95 &     597637 &     282.97 \\
\hline
 \multicolumn{1}{c}{\multirow{ 2}{*}{300New\_5}}&         GA & \multicolumn{ 1}{c}{{\multirow{ 2}{*}{3}}} & {0.04} &  {152} & {64.99} & {62704} & {65.05} \\

 &         AA & \multicolumn{ 1}{c}{{}} & { 0.04} &  {152} & {379.84} & {402154} & {379.88} \\
\hline
\end{tabular}  

 \caption{DC-OTS Results with Feasibility Oriented setting and inequalities \eqref{eq:fEqualIneqDesc} are  separated.}
   \label{tab:conv}
 \end{table}

We  summarize our observations as below:

\begin{itemize}
    \item Feasibility oriented (12 unsolved) is much better than Default (24 unsolved). All other statistics are improved as well.
    
    \item
    The setting with cuts separated from the convex hull and the setting with inequalities separated from \eqref{eq:fEqualIneqDesc} are better than the setting with feasibility oriented (they both solve 5 more instances). Most of the statistics are improved (especially for the harder instances 118Blumsack\_9G, 118Blumsack\_15p, 300New\_5).

    \item
    The overall performances of the setting with cuts separated from the convex hull and the setting with inequalities separated from \eqref{eq:fEqualIneqDesc} are similar.
    \begin{itemize}
        \item     The latter adds less cuts: This is more advantageous in easier instances 118Blumsack\_15 and 118Blumsack\_5p.
\item     The former adds more cuts: This is more advantageous in harder instances 118Blumsack\_9G and 118Blumsack\_15p.
\item     For 300New\_5 instances, the former adds significantly more cuts, which seem to result in a heavier and slower model (one more unsolved).
    \end{itemize}
    
    \item
\textcolor{black}{
The setting with inequalities separated from \eqref{eq:fEqualIneqDesc} does not improve the percentage gap closed, defined as
\[
100 \times \frac{UB - LP'}{UB - LP},
\]
except for three instances from 118Blumsack\_9G with negligible improvements (here, $UB$ is the best upper bound found, $LP$ is the lower bound obtained from the LP relaxation and $LP'$ is the lower bound obtained from the LP relaxation with the added inequalities). 
The setting with cuts separated from the convex hull is slightly more successful with average percentage gap closed reported as 0.08\%,  0.38\%, 0.08\%, 0.00\% and 0.16\% respectively for the five instance families.  The OTS Problem is known to have dual degeneracy (the cuts separate the current fractional point but there exist other optimal solutions of the LP with the same objective function value) and these results are not very surprising in this respect.
}

\end{itemize}








\section{Conclusion}\label{sec:conc}
We consider a substructure in DC-OTS representing a single bus node and the lines connected to it. We show that optimizing over this substructure is NP-Hard, and we propose an extended formulation for its convex hull and show when this convex hull is of a compact size. We also propose an outer approximation for its convex hull, which is more compact than our extended formulation, and we show that in a special case of the problem, the outer approximation is exactly the convex hull. In this case, we are able to obtain the convex hull in the original space. Computationally, these cuts are able to help solve more instances in lesser time over the commercial solver CPLEX.

\textcolor{black}{We note that while the substructure we studied is motivated from DC-OTS, this set may also be obtained as a substructure in the context of AC-OPF. For example, we can consider a single bus and lines connected to it, together with the real power flowing in these line. It may also be possible to use the results of this study to other network based applications like gas and water networks.}



\bibliographystyle{plain}
\bibliography{references}

\end{document}